\documentclass[a4paper,12pt]{amsart}

\usepackage{amsmath,amsfonts,amsthm,amssymb,dsfont}
\usepackage[alphabetic]{amsrefs}

\usepackage{times}

\usepackage{enumerate}

\usepackage{mathrsfs}
\usepackage{enumerate, xspace}
\usepackage{graphicx}

\usepackage{amssymb,amsmath,amsthm}

\usepackage{verbatim}

\usepackage{caption}
\usepackage{subcaption}

\DeclareMathOperator*{\freeproduct}{\raisebox{-0.2ex}{\scalebox{1.4}{$\ast$}}}

\usepackage[pdfauthor={Dominik Gruber, Alexandre Martin, and Markus Steenbock},pdftitle={Finite index subgroups without unique product in graphical small
cancellation groups},pdfkeywords={Graphical small cancellation theory, hyperbolic groups, unique product property, free products of groups},pdftex]{hyperref}

\hypersetup{colorlinks,%
citecolor=black,%
filecolor=black,%
linkcolor=black,%
urlcolor=black,%
}

\theoremstyle{definition}		\newtheorem*{df}{Definition}

\theoremstyle{plain}		\newtheorem{prop}{Proposition}		
				
				\newtheorem*{lem}{Lemma}
				\newtheorem*{thm}{Theorem}

\theoremstyle{remark}

\begin{document}

\author{D. Gruber}
\address{Universit\"at Wien, Fakult\"at f\"ur Mathematik\\
Oskar-Morgenstern-Platz 1, 1090 Wien, Austria.
}
\email{dominik.gruber@univie.ac.at}
\email{alexandre.martin@univie.ac.at}
\email{markus.steenbock@univie.ac.at}
\author{A. Martin}

\author{M. Steenbock}


\title[Finite index subgroups in graphical small cancellation groups]{Finite index subgroups without unique product in graphical small
cancellation groups}

\begin{abstract}
 We construct torsion-free hyperbolic groups without unique product whose subgroups up to some given finite index are themselves non-unique product groups. This is achieved by generalising a construction of Comerford to graphical small cancellation presentations, showing that for every subgroup $H$ of a graphical small cancellation group there exists a free group $F$ such that $H*F$  admits a graphical small cancellation presentation.
\end{abstract}

\maketitle

The unique product property was introduced as a way to prove Kaplansky's zero-divisor conjecture \cite{kaplansky} on the group ring of a torsion-free group \cite{cohen}. A group $G$ is said to have the \emph{unique product property} if every pair of non-empty finite sets $A,B$ of $G$ admits a \emph{unique product}, that is, if there exists $c\in G$ for which there exist unique elements $a\in A$ and $b\in B$ satisfying $c=ab$. Delzant showed \cite{delzant_sur_1997} that every \emph{residually finite} torsion-free hyperbolic group admits a finite index subgroup with the unique product property. It is still unknown whether every hyperbolic group is residually finite. In light of the above, the existence of a torsion-free hyperbolic group all of whose finite index subgroups are non-unique product would provide an example of a non-residually finite hyperbolic group. In \cite{markus2} Arzhantseva--Steenbock use a version of the Rips construction \cite{rips_subgroups_1982} to produce explicit hyperbolic groups that 
have a non-unique product subgroup of some given finite index. They ask whether there exist torsion-free hyperbolic groups all of whose subgroups \emph{up to} some given finite index are non-unique product groups. In this note, we answer this question in the positive.

The first torsion-free groups without the unique product property are due to Rips--Segev \cite{rips_torsion-free_1987}. Such groups can be realised as finitely presented graphical small cancellation groups over a free product of torsion-free hyperbolic groups \cite{markus}, thus providing the first examples of torsion-free \emph{hyperbolic} non-unique product groups. Very little is known about their residual properties or the properties of their subgroups.
We construct Rips--Segev groups which have many finite index subgroups without the unique product property. More precisely, we prove the following:

\begin{thm} Let $k\geqslant 1$ be an integer. There exists a torsion-free hyperbolic group $G$ without the unique product property such that for all $1 \leqslant h\leqslant k$:
\begin{enumerate}
 \item there exists a subgroup
 of index $h$;
 \item every subgroup of index $h$ is a non-unique product group.
 \end{enumerate} 
\end{thm}

Our proof together with \cite{markus} can be used to construct an explicit presentation of $G$. In the course of our proof, we provide 
 a generalisation of a construction of Comerford \cite{comerford_subgroups_1978} to graphical small cancellation presentations, which is of independent interest.
Given a small cancellation presentation of a group $G$ and an index $h$ subgroup $H$, it provides an explicit small cancellation presentation for $H*F_{h-1}$, where $F_{h-1}$ is the free group of rank $h-1$.

 \medskip
 
\noindent
{\bf Acknowledgments.} We thank Goulnara Arzhantseva for encouraging us to write this note. D. Gruber and A. Martin are supported by the ERC grant ANALYTIC no.\  259527 of G. Arzhantseva. M. Steenbock is recipient of the DOC fellowship of the Austrian Academy of Sciences and is partially supported by the ERC grant ANALYTIC no.\  259527 of G. Arzhantseva.

\section{Comerford construction for graphical small cancellation}

We extend the aforementioned construction of Comerford \cite{comerford_subgroups_1978} for classical small cancellation presentations \cite{lyndon_combinatorial_1977} to graphical small cancellation presentations as considered in \cites{ollivier_small_2006,dominik,markus}. 

Let $\Gamma$ be a graph. A \emph{labelling} of $\Gamma$ by a set $S$ is a choice of orientation on each edge and a map assigning to each edge an element of $S$, called \emph{label}. Given an 
edge-path $p$ on $\Gamma$, the \emph{label of $p$}, denoted $\omega(p)$ is the product of the labels of the edges traversed by $p$ in the free monoind on $S\sqcup S^{-1}$. Here a letter is given exponent $+1$ if the corresponding edge is traversed in its direction and exponent $-1$ if it is traversed in the opposite direction. A graph labelled by a set $S$ defines a group $G(\Gamma)$ given by the following presentation:
$$\langle S\mid\Gamma\rangle:=\langle S\mid\text{labels of simple closed paths on }\Gamma\rangle.$$

A labelling is \emph{reduced} if the labels of immersed 
paths are freely reduced words. A \emph{piece} with respect to a labelled graph $\Gamma$ is a labelled path $p$ that has two \emph{essentially distinct} immersions $\iota_1,\iota_2:p\to \Gamma$. Here essentially distinct means that there does not exist a label-preserving automorphism $\phi:\Gamma\to\Gamma$ such that $\iota_1=\phi\circ\iota_2$.

\begin{df}
 A labelled graph $\Gamma$ satisfies the \emph{$Gr(p)$ small cancellation condition} for $p\in\mathbb N$  if
\begin{itemize}
\item the labelling is reduced and
\item no nontrivial simple closed path is the concatenation of fewer than $p$ pieces.
\end{itemize}
\end{df}

Let $S$ be a set, and denote by $F(S)$ the free group on $S$. Let $\ell$ be a length function on $F(S)$. Examples of such length functions are:
The \emph{word length}, which counts the number of generators in a reduced word. 
The \emph{free product length} (or syllable length \cite{lyndon_combinatorial_1977}*{Ch.V.9}). Given a partition $\Pi$ of $S$, the free product length counts the number of factors in the normal form of an element with respect to $F(S)=\freeproduct_{P\in \Pi} F(P)$.

\begin{df}
A  labelled graph $\Gamma$ satisfies the \emph{$Gr_{\ell}'(\lambda)$ small cancellation condition} for  $\lambda>0$  if
\begin{itemize}
\item the labelling is reduced and
\item every piece $p$ that is subpath of a nontrivial simple closed path $\gamma$ satisfies $\ell(\omega(p))<\lambda\ell(\omega(\gamma))$.
\end{itemize}
Denote by $Gr'(\lambda)$ the $Gr'_\ell(\lambda)$-condition where $\ell$ is the word length. Given a partition of $S$, denote by $Gr'_*(\lambda)$ the $Gr'_\ell(\lambda)$-condition where $\ell$ is the associated free product length.
\end{df}

 If $\Gamma$ satisfies the $Gr_\ell'(\lambda)$-condition for $\lambda\leqslant\frac{1}{p-1}$, then $\Gamma$ satisfies the $Gr(p)$-condition. 

Metric graphical small cancellation with respect to the word length of the free group was first studied in \cite{ollivier_small_2006}. Non-metric graphical small cancellation over free groups was first studied in \cite{dominik}. Metric graphical small cancellation over arbitrary free products was first studied in \cite{markus}.

\begin{prop}\label{ComerfordGraphical}
 Let $p\in\mathbb N$ and $\lambda>0$. Let $\Gamma$ be a graph labelled by a set $S$, and let $H$ be a subgroup of index $h$ (finite or infinite) in $G(\Gamma)$. Then there exist 
 a graph $\Gamma_H$ labelled by $S\times(G(\Gamma)/H)$ such that $G(\Gamma_H)=H*F_{h-1}$, where $F_{h-1}$ is the free group of rank $h-1$, and such that:
 \begin{itemize}
  \item If $\Gamma$ satisfies the $Gr(p)$-condition, then so does $\Gamma_H$.
  \item If $\Gamma$ satisfies the $Gr'(\lambda)$-condition, then so does $\Gamma_H$.
  \item If $\Gamma$ satisfies the $Gr'_*(\lambda)$-condition with respect to $F(S)=\freeproduct_{P\in \Pi}F(P)$, where $\Pi$ is a partition of $S$, then $\Gamma_H$ satisfies the $Gr'_*(\lambda)$-condition with respect to $F(S\times G(\Gamma)/H)=\freeproduct_{P\in \Pi}F\bigl(P\times G(\Gamma)/H\bigr)$.
 \end{itemize}

\end{prop}

\begin{proof} We extend the proof of Comerford \cite{comerford_subgroups_1978}. Denote by $K$ the labelled oriented graph that has a single vertex and for each $s\in S$ a single oriented edge labelled $s$. In each component $\Gamma^i$ of $\Gamma$ fix a basepoint. The labelling of $\Gamma$ by $S$ can be viewed as a basepoint-preserving graph homomorphism
$$\omega:\Gamma\to K.$$ 
A space $X$ with fundamental group $G(\Gamma)$ is obtained by attaching, for each component $\Gamma^i$ of $\Gamma$, the topological cone $C\Gamma^i$ over $\Gamma^i$ onto $K$ along the map $\omega$.

Let $H$ be a subgroup of index $h$ (finite or infinite) in $G$, and denote 
$$S_H:=S\times G/H.$$
For simplicity, we write an ordered pair $(s,v)$ as $s_v$.
We now construct a graph $\Gamma_H$ labelled by $S_H$ such that $G(\Gamma_H)= H*F_{h-1}$.

Let $\pi_H:X_H\to X$ be a connected cover with $\pi_1(X_H)=H$. Then $\pi^{-1}_H(K)$ is a Schreier coset graph of $H\leq G(\Gamma)$, and, in particular, every vertex of $\pi^{-1}_H(K)$ is an element of $G(\Gamma)/H$. The map $\pi^{-1}_H(K)\to K$ is a labelling of $\pi^{-1}_H(K)$. We construct a new labelling of $\pi^{-1}_H(K)$ over $S_H$ as follows: We do not change orientations of edges. We replace the label of every edge starting at a vertex $v$ and labelled by $s$ by the label $s_v$. Denote the resulting labelled graph by $K_H$ and its labelling function~by~$\omega_H$.

Recall that we fixed basepoints in the components $\Gamma^i$ of $\Gamma$ and that the topological cone over each $\Gamma^i$ is simply connected. Thus, for each vertex $v\in K_H$, there exists a graph homomorphism $\omega_v:\Gamma\to K_H$ taking all basepoints to $v$. We interpret this homomorphism as  labelling $\omega_v$ on $\Gamma$. Denote the graph $\Gamma$ with the labelling $\omega_v$ by $\Gamma_v$ and denote
$$\Gamma_H:=\bigsqcup_{v\in G(\Gamma)/H} \Gamma_v.$$

We show that that $G(\Gamma_H) = H*F_{h-1}$: In $X_H$, identify all vertices in $\pi_H^{-1}(K)$ and denote the resulting space by $X_H^*$. We compute the fundamental group of $X_H^*$ as follows: Consider the disjoint union of $X_H$ and a space consisting of a single vertex $b$. Now add edges connecting $b$ to every vertex of $\pi_H^{-1}(K)$. The fundamental group of this space is $H*F_{h-1}$, and it is homotopy equivalent to $X_H^*$. Therefore, $X_H^*$ has fundamental group $H*F_{h-1}$.

Consider $X_H^*$ with the labelling of edges induced from $K_H$. The image of $K_H$ in $X_H^*$ has a single vertex and for each $s_v\in S_H$ a single oriented edge labelled $s_v$. $X_H^*$ is obtained by attaching the topological cone over each component of $\Gamma_H$ along the labelling map. Thus,
$$G(\Gamma_H)=\pi_1( X_H^*)=H*F_{h-1}.$$

The (not label-preserving) maps of labelled graphs $\pi_v:\Gamma_v\to\Gamma$ induced by the identity on the underlying graphs are isometries with respect to the length functions we consider. The labelling of each $\Gamma_v$ is reduced if the labelling of $\Gamma$ is. We show that every piece in $\Gamma_H$ maps to a piece in $\Gamma$ via a map $\pi_v$. Since nontrivial simple closed paths map to nontrivial simple closed paths, this is sufficient to show that $\Gamma_H$ satisfies the claimed small cancellation conditions if $\Gamma$ does.

We start by an observation: Let $e$ be an edge in a component $\Gamma^i$ of $\Gamma$, and let $s = \omega(e)$. Let $v \in G/H$. By the unique lifting property of covering spaces, there exists a unique lift of the map $C\Gamma^i \to X$ (induced by $\omega: \Gamma \to K$) to $X_H$ which sends $e$ to the edge labelled $s_v$, and thus a unique lift of $\Gamma^i \to K$ to $K_H$ sending $e$ to the edge labelled $s_v$.

Now let $\iota_1: p\to\Gamma^i \subset \Gamma$ and $\iota_2:p\to\Gamma^j\subset \Gamma$ be two immersions of a non-trivial path $p$ into $\Gamma$ and $v, w \in G/H$ such that the labellings $\omega_v \circ \iota_1, \omega_w \circ \iota_2 : p \to K_H$ by $S_H$ coincide.  Assume that there exists an $\omega$-preserving automorphism $\phi$ of $\Gamma$ such that $\phi \circ \iota_1=\iota_2$. Let $e$ be an edge of $p$. By construction, the maps $\omega_w\circ\phi$ and $\omega_v$ are two lifts of $\omega: \Gamma^i \to K$ to $K_H$ that coincide on $\iota_1(e)$, hence they are equal by the above observation. Thus, $\phi$ induces an isomorphism from $\Gamma^i$ to $\Gamma^j$ that is compatible with labellings $\omega_v$ of $\Gamma^j$ and $\omega_w$ of $\Gamma^j$. This can be extended to a label-preserving automorphism $\phi_H: \Gamma_H \to \Gamma_H$ by sending $\Gamma^j$ with labelling $\omega_w$ to $\Gamma^i$ with labelling $\omega_v$ by means of $\phi^{-1}$, and by being the identity on every other labelled 
component. Thus, if $\iota_1 : p\to\Gamma_v \subset \Gamma_H$ and $\iota_2: p\to\Gamma_w \subset \Gamma_H$ are two essentially distinct paths in $\Gamma_H$ (with respect to the labelling $\omega_H$) which have the same labels, then $\pi_v\circ\iota_1:p\to\Gamma$ and $\pi_w\circ\iota_2:p\to\Gamma$ are essentially distinct in $\Gamma$ (with respect to the labelling $\omega$). 
\end{proof}

\section{Groups without unique product}

The first  construction of torsion-free groups without the unique product property is due to \cite{rips_torsion-free_1987}.
We present here a generalisation of this  construction, following \cite{markus}, which allows more flexibility in the choice of generators and relators in the presentations under consideration. This will be used to prove our main theorem.

Let $F(S)$ and $F(T)$ be free groups over non-empty distinct sets $S$ and $T$.  We start by constructing a graph $\Gamma$ labelled by $S\sqcup T$ which will be used to define non-unique product groups. This is done in three steps. 

Choose non-trivial cyclically reduced elements $a \in F(S)$ and $b\in F(T)$. Let $N \geqslant 1$ be an integer and choose integers $C_1,\ldots,C_{N} \geqslant 1$. For each $1 \leqslant i \leqslant N$, let $p_i$ be the oriented line graph labelled by $S$ whose label is $a^{C_i}$. Denote by $u_{i,j}$ the terminal vertex of the initial subpath labelled $a^{j}$. Let $p_b$ be the oriented line graph labelled by $T$ whose label is $b$. Denote the initial vertex of $p_b$ by $v_0$ and the terminal vertex by $v_1$.

For every $1 \leqslant i \leqslant N$, we now construct a new graph $p_i'$ out of $p_i$ as follows. Consider $C_i+1$-many copies of $p_b$, denoted $(p_b)_{i, 0}, \ldots, (p_b)_{i, C_i}$.  We construct the graph $p_i'$ from the disjoint union of $p_i$ and the various $(p_b)_{i,j}, 0 \leqslant j \leqslant C_i$, by identifying the vertex $u_{i, j}$ of $p_i$ with the vertex $(v_0)_{i,j}$ of $(p_b)_{i,j}$ for every $0 \leqslant j \leqslant C_i$. Each $p_i'$ naturally comes with a labelling by $S\sqcup T$.

We now define the graph $\Gamma$ from the disjoint union of the labelled graphs $p_i', 1 \leqslant i \leqslant N$ as follows.  For each $1 \leqslant i \leqslant N$, choose  four integers $1\leqslant
N_{i,1},N_{i,2},N_{i,3}$, $N_{i,4}\leqslant N$ and for each $1 \leqslant j \leqslant 4$, an integer  $0\leqslant P_{i,j}\leqslant C_{N_{i,j}}$. We identify the vertex $u_{i,0}$ (respectively $(v_{1})_{i,0}$,
$u_{i,C_i}$, $(v_{1})_{i, C_i}$) with the vertex
$(v_{1})_{N_{i,1},P_{i,1}}$ (respectively $u_{N_{i,2},P_{i,2}}$, $(v_{1})_{N_{i,3},P_{i,3}}$, $u_{N_{i,4},P_{i,4}}$). As before, $\Gamma$ naturally inherits a labelling by $S\sqcup T$. 

\medskip

Note that $\Gamma$ depends on the various choices of $a,b, N, (C_i),  (N_{i,j})$ and $(P_{i,j})$. We will denote it  $\Gamma \big(a,b, N, (C_i), (N_{i,j}), (P_{i,j})\big)$ when emphasising this dependence.

\begin{df}
The graph $\Gamma= \Gamma\big(a,b, N, (C_i), (N_{i,j}), (P_{i,j}) \big)$ is called the \textit{Rips--Segev graph (over $F(S)*F(T)$)} associated to the \textit{coefficient system} $\big(a$, $b$, $N$, $(C_i)$, $(N_{i,j})$, $(P_{i,j}) \big).$
\end{df}

Combinatorial considerations of graphs with large girth yield the following existence result:

\begin{prop}[\cite{markus}] 
For all non-trivial cyclically reduced $a\in F(S)$ and $b\in F(T)$, there exists an explicit choice of coefficients such that the associated Rips--Segev graph is connected and satisfies the $Gr_{*}'(\frac{1}{6})$-condition with respect to the free product length  on $F(S)*F(T)$.\qed \label{mainRipsSegev2}
\end{prop}

Consider a connected Rips--Segev graph $\Gamma=\Gamma\big(a,b$, $N$, $(C_i)$, $(N_{i,j})$, $(P_{i,j})\big)$. We now construct non-empty finite subsets of elements of $F(S)*F(T)$. For $1 \leqslant i \leqslant N$, choose a path $\gamma_i$ in $\Gamma$ from $u_{1,0}$ to $u_{i,0}$ and let $w_i$ be the label of $\gamma_i$ in $F(S)*F(T)$. For each $1 \leqslant i \leqslant N$, we define the following subsets of $F(S)*F(T)$:
\[A_i:=\{w_{i}, w_{i}a,w_{i}a^2, \ldots, w_{i}a^{C_i-1}\}.\]
Finally, let
$$A:= \bigcup_{1 \leqslant i \leqslant N} A_i\hspace{12pt}\text{ and }\hspace{12pt}B:= \{1, a, b, ab \}.$$

In presence of graphical small cancellation conditions, the image of $A$ and $B$ in $G(\Gamma)$ define non-empty finite subsets without a unique product. More precisely, we have the following fundamental results about Rips--Segev graphs:

\begin{prop}[\cite{markus}]
Let $\Gamma$ be a finite labelled graph over $F(S)*F(T)$ which is a non-empty disjoint union of connected  Rips--Segev graphs over $F(S)*F(T)$. If $\Gamma$ satisfies the $Gr_{*}'(\frac{1}{6})$-condition with respect to the free product length on $F(S)*F(T)$, then $G(\Gamma)$ is torsion-free hyperbolic and does not have the unique product property. 
\qed
\label{mainRipsSegev}
\end{prop}

The proof uses the following arguments: Results on $Gr_*'(\frac{1}{6})$-presentations over free products \cite{markus}, or, alternatively, $Gr(7)$-presentations over free groups \cite{dominik} yield that $G(\Gamma)$ is torsion-free hyperbolic and that every component of $\Gamma$ injects into the Cayley graph of $G(\Gamma)$. Consider a component $\Gamma^i$ of $\Gamma$. Since $\Gamma^i$ injects into the Cayley graph, the sets $A$ and $B$ associated to $\Gamma^i$ inject into $G(\Gamma)$ under the projection $F(S)*F(T) \to G(\Gamma)$. The labelled paths on $\Gamma^i$ give rise to more than one way of writing each element in $AB$ as product of elements of $A$ and $B$, therefore ensuring the non-unique product property. 
 A direct proof that $A$ and $B$ embed can be found in \cite{markus2}, again using the graphical small cancellation over  free products.

We now move to the proof of our main theorem. Fix an integer $k \geqslant 1$. In the above notation let $S:=\{s\}$ and $T:=\{t\}$.  Set $$a:=s^{k!}, b:=t^{k!}.$$ By Proposition \ref{mainRipsSegev2}, we can find coefficients $\big( N, (C_i), (N_{i,j}), (P_{i,j})\big)$ such that the associated Rips--Segev graph $\Gamma:=\Gamma\big(a$, $b$, $N$, $(C_i)$, $(N_{i,j})$, $(P_{i,j})\big)$ is connected and satisfies the $Gr_{*}'(\frac{1}{6})$-condition with respect to the free product length on $F(\{s\})*F(\{t\})$.  We now show that $G:=G(\Gamma)$ is a group for $k$ as claimed in our main theorem.

\begin{lem} Let $Q$ be a 2-generated group of cardinality $h\leqslant k$. Then $G$ admits a surjective homomorphism to $Q$. 
\end{lem}

\begin{proof} Let $\{s',t'\}$ be a generating set for $Q$. Since $Q$ has cardinality $h$, $s'$ and $t'$ both have order dividing $k!$. By construction, every defining relator of $G$ (that is, every label of a cycle of $\Gamma$) is a product of powers of $s^{k!}$ and $t^{k!}$. Thus, the surjective map $F(\{s\}) *F(\{t\})\to Q$ sending $s$ and $t$ to $s'$ and $t'$ respectively, sends the defining relators of $G$ to the identity. This yields a surjective homomorphism $G\to Q$. 
\end{proof}

\begin{proof}[Proof of the main theorem]
Let $h\leqslant k$ and $H$ a subgroup of $G$ of index $h$. We use the same notations as in the proof of Proposition \ref{ComerfordGraphical}. Recall that $\Gamma_H = \bigsqcup_{v \in G/H} \Gamma_v$, where $ G/H$ is the set of vertices of $K_H$, and each $\Gamma_{v}$ is isomorphic to $\Gamma$ as an unlabelled oriented graph.

\begin{figure}
\includegraphics{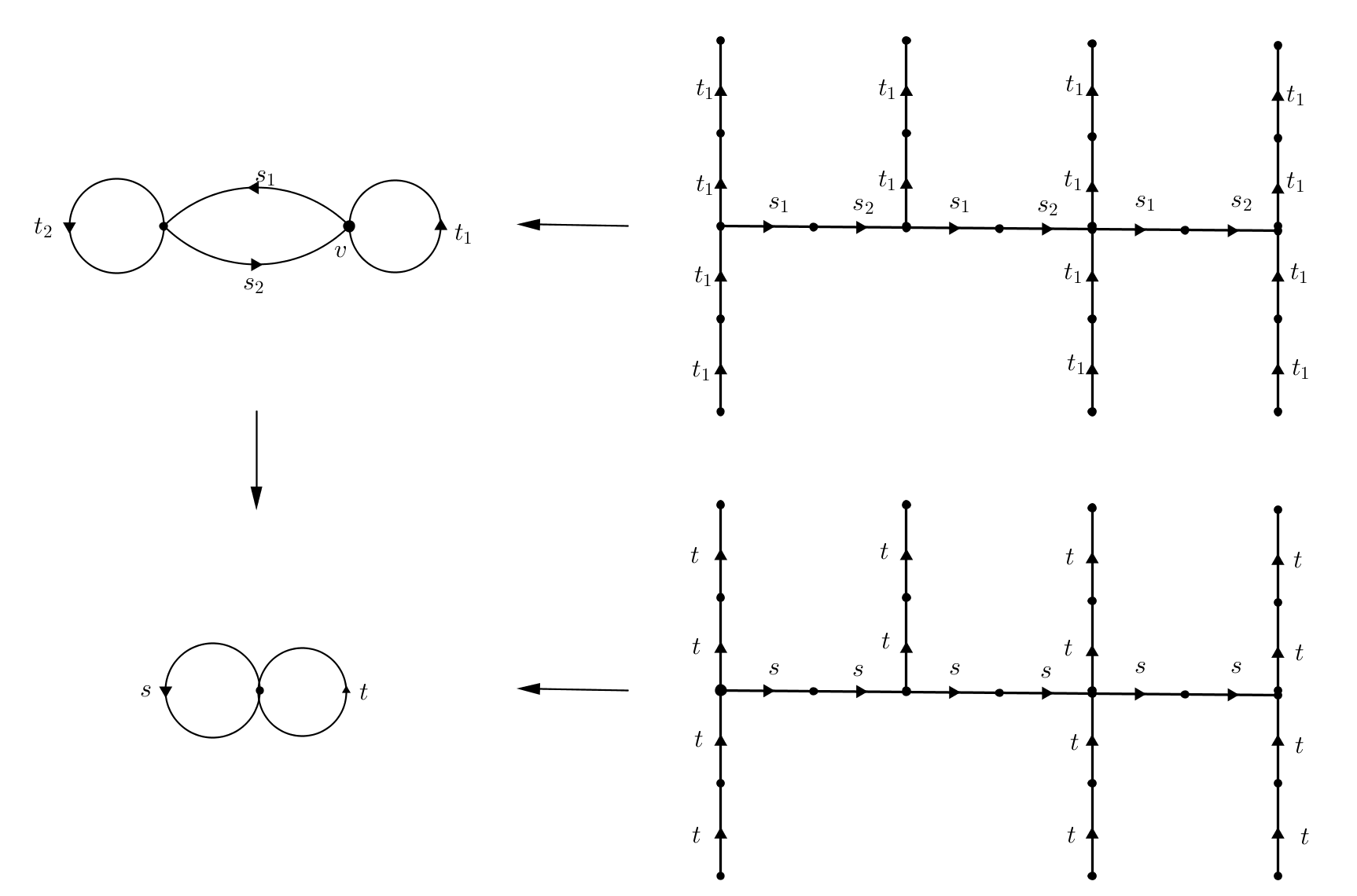}
\caption{The situation for an index 2 subgroup in the case $a=s^2$, $b=t^2$. 
Upper left: $K_H$, upper right: a
part of $\Gamma_v\subseteq\Gamma_H$, lower left: $K$, lower right: a part of $\Gamma$.}
\label{figure}
\end{figure}

For each $v \in G/H$, the connected component of the preimage under $\pi_H: K_H \rightarrow K$ of the oriented edge labelled $s$ (respectively $t$) containing $v$ is an oriented cycle $\alpha_v$ (respectively $\beta_v$) labelled by  $\{s\} \times G/H$ (respectively $\{t\} \times G/H$) of length at most $k$ (see Figure \ref{figure}). Define $$a_v:=\omega_H(\alpha_v)^{k!/|\alpha_v|}\text{ and }b_v:=\omega_H(\beta_v)^{k!/|\beta_v|};$$ here $|\, .\, |$ denotes the edge-length of paths in $\Gamma$ or $\Gamma_H$ respectively.

Thus, the map of labelled graphs $\Gamma \rightarrow \Gamma_{v}$ induced by the identity on the underlying graph sends every path of $\Gamma$ with label $a$ that starts at some $u_{i,j}$ to a path of $\Gamma_{v}$ with label $a_v$, and every path with label $b$ starting at some $(v_0)_{i,j}$ to a path of $\Gamma_{v}$ with label~$b_v$. Therefore, the graph $\Gamma_{v}$ is the Rips--Segev graph over $F( \{s\} \times G/H ) * F( \{t\} \times G/H)$ with coefficient system $\big(a_v, b_v, N, (C_i), (N_{i,j}), (P_{i,j}) \big)$.

By Proposition \ref{ComerfordGraphical}, the labelling of  $\Gamma_H=\bigsqcup_v \Gamma_v$ satisfies the $Gr_{*}'(\frac{1}{6})$-condition with respect to the free product length on $F( \{s\} \times G/H ) * F( \{t\} \times G/H)$. 
Thus, $G(\Gamma_H)=H*F_{h-1}$ does not satisfy the unique product property by Proposition \ref{mainRipsSegev}. As the unique product property is stable under free products, it follows from the fact that free groups are unique product groups that $H$ does not have the unique product property.
\end{proof}


\end{document}